\documentclass{article}

\usepackage[margin=1in]{geometry}
\usepackage[numbers]{natbib}
\usepackage[utf8]{inputenc}
\usepackage[T1]{fontenc}
\usepackage[colorlinks,linkcolor=blue,citecolor=black,urlcolor=blue]{hyperref}
\usepackage{url}
\usepackage{amsmath,amssymb,amsfonts,amsthm,mathtools}
\usepackage{booktabs}
\usepackage{graphicx}
\usepackage{caption}
\usepackage{subcaption}
\usepackage{algorithm}
\usepackage{algpseudocode}
\usepackage{float}

\newtheorem{assumption}{Assumption}
\newtheorem{theorem}{Theorem}
\newtheorem{remark}{Remark}
\newtheorem{lemma}{Lemma}

\newcommand{\R}{\mathbb{R}}
\newcommand{\St}{\mathrm{St}}
\newcommand{\sym}{\operatorname{sym}}
\newcommand{\revise}[1]{{\color{black}#1}}

\allowdisplaybreaks

\title{A Retraction-Free EXTRA Method for Decentralized Optimization on the Stiefel Manifold}
\author{
\begin{minipage}[t]{0.48\textwidth}
\centering
Shu Li\\
Yau Mathematical Sciences Center, Tsinghua University, Beijing, China\\
\href{mailto:leo857290@gmail.com}{\texttt{leo857290@gmail.com}}
\end{minipage}
\hfill
\begin{minipage}[t]{0.48\textwidth}
\centering
Jiang Hu\thanks{Corresponding author.}\\
Yau Mathematical Sciences Center, Tsinghua University, Beijing, China\\
\href{mailto:jianghu@tsinghua.edu.cn}{\texttt{jianghu@tsinghua.edu.cn}}
\end{minipage}
}
\date{}

\begin{document}
\setlength{\parindent}{0pt}
\setlength{\parskip}{0.5em}
\maketitle

\begin{abstract}
Decentralized optimization provides a fundamental framework for large-scale learning and signal processing with distributed data. We study decentralized optimization with orthogonality constraints on the Stiefel manifold and propose RF-EXTRA, a distributed retraction-free primal-dual method on static undirected networks. The method combines an approximate gradient mapping for orthogonality-constrained optimization with an EXTRA-based decentralized recursion, thereby avoiding retractions while preserving a simple communication pattern. On the theoretical side, the analysis considers \revise{the joint error} $(\mathbf X_k-\overline{\mathbf X}_k,\mathbf s_k-\overline{\mathbf s}_k)$ \revise{in the local variables and local directions}, and establishes a contractive recursion for the joint error. 
\revise{This contractivity ensures that the joint error can be controlled using small yet constant step sizes, thus leading to an exact $\mathcal O(1/K)$ convergence rate of RF-EXTRA to a stationary point.} Experiments on PCA and low-rank matrix completion show that RF-EXTRA compares favorably with the reported decentralized baselines and exhibits strong communication efficiency on the tested tasks on the Stiefel manifold.
\end{abstract}



\section{Introduction}
\label{sec:introduction}

Decentralized optimization \cite{tsitsiklis1986distributed,nedic2009distributed,yuan2016convergence} studies how a network of agents cooperatively solves a global problem using local data and local communication. It has become a fundamental framework in large-scale learning \cite{alghunaim2022unified,yuan2018exact}, signal processing \cite{shi2015extra}, and control \cite{qiu2016distributed,duchi2011dual}, where data are naturally distributed and centralized aggregation is either costly or undesirable. \revise{In the Euclidean (i.e., unconstrained) decentralized optimization,} the goal is to combine computational scalability with communication efficiency while preserving the solution quality of the underlying centralized model.

Many decentralized problems are not \revise{Euclidean or unconstrained}. In distributed principal component analysis, subspace estimation, and related matrix-learning tasks, the decision variable is required to have orthonormal columns, and the optimization therefore takes place on the Stiefel manifold
\[
\St(d,r)=\{X\in\R^{d\times r}: X^\top X=I_r\}.
\]
Such orthogonality constraints also arise in deep models with orthogonal layers \cite{Huang2018OWN}, normalization mechanisms \cite{Cho2017RiemannianBN}, and low-rank fine-tuning schemes \cite{Zhang2024RiemannianLoRA,zhang2024retraction}. Accordingly, decentralized optimization on the Stiefel manifold has become an important problem in decentralized learning and matrix optimization.

\revise{Motivated by these applications, this paper studies the consensus-constrained optimization problem on the Stiefel manifold}
\revise{\[
\min_{X_1,\dots,X_n\in\St(d,r)} \frac1n\sum_{i=1}^n f_i(X_i),
\qquad \text{s.t.}\qquad X_1=\cdots=X_n.
\]}
\revise{where each agent $i\in\{1,\dots,n\}$ has access only to its local smooth objective $f_i$ and communicates over a fixed connected undirected graph. Our aim is to develop a decentralized method that preserves the correction mechanism of EXTRA \cite{shi2015extra} while using a retraction-free update on the Stiefel manifold. This direction is algorithmically attractive because EXTRA-type methods achieve exact decentralized optimization with only local communication, and extending such a communication-efficient mechanism to nonconvex constrained problems on the Stiefel manifold would substantially broaden its scope. It is also computationally appealing because modern GPUs are highly optimized for basic linear algebra subprograms (BLAS) \cite{blackford2002updated}, whereas retraction operations such as repeated QR or polar orthogonalization are typically less GPU-friendly than simple matrix multiplications and additions.} \revise{The main analytical difficulty is that primal-dual methods are already delicate to analyze under nonconvex constraints, where the constraint geometry interacts with stationarity and feasibility in a coupled way. This difficulty becomes sharper for EXTRA-type schemes: although EXTRA admits a primal-dual interpretation, its correction is implemented through an inexact recursive update rather than an exact primal-dual step. Combining such an inexact primal-dual recursion with Stiefel manifold constraints and a retraction-free feasibility mechanism makes the convergence analysis highly nontrivial. This naturally gives rise to the following question:}

\begin{center}
\revise{\emph{Can we design a retraction-free EXTRA algorithm for decentralized optimization over the Stiefel manifold with provable convergence guarantee?}}
\end{center}

\subsection{Related Work}
\label{subsec:related-work}

\subsubsection{Decentralized optimization in Euclidean space}

The development of decentralized optimization began with primal methods such as decentralized gradient descent (DGD) \cite{yuan2016convergence} and decentralized SGD type variants for stochastic objectives \cite{koloskova2021improved,Tang2018D2DT}, which established the basic consensus optimization framework but generally suffer from steady state bias under constant stepsizes. To remove that bias, later primal-dual or correction based methods introduced auxiliary states that couple consensus and optimization more tightly. Representative milestones include gradient tracking and DGT type methods \cite{Nedic2016AchievingGC,Qu2017HarnessingSA,Li2019SDIGingAS,pu2021distributed,Liu2023DecentralizedGT}, EXTRA \cite{shi2015extra}, exact diffusion \cite{yuan2018exact}, NIDS \cite{Li2019NIDS}, and prox-PDA \cite{hong2017prox}. These methods clarified how exactness can be preserved over static networks while maintaining low per-iteration communication cost.

The surrounding Euclidean literature further expanded this line through refined analyses for heterogeneous and nonconvex settings \cite{Yuan2021RemovingDH,alghunaim2022unified,di2016next}, compressed communication \cite{Koloskova2019DecentralizedSO,He2023UnbiasedCS,chen2026greedy,beznosikov2023biased,Qian2020ErrorCD}, time varying and directed networks \cite{Nedic2016AchievingGC,Liang2023UnderstandingTI}, local step variants \cite{Alghunaim2023LocalEF}, and graph design viewpoints \cite{Ying2021ExponentialGI}. \revise{These Euclidean results provide the algorithmic backbone for RF-EXTRA, but they do not address the feasibility and geometric difficulties introduced by Stiefel manifold constraints; this motivates adapting the EXTRA correction mechanism to the Stiefel manifold setting.}

\subsubsection{Distributed optimization on the Stiefel manifold}

\revise{Decentralized optimization over the Stiefel manifold has attracted growing attention.} DRSGD and DRGTA \cite{Chen2021DecentralizedRG} \revise{study} decentralized Riemannian gradient descent on the Stiefel manifold and \revise{show} that distributed PCA type tasks can be handled directly in the manifold setting. More recently, DPRGD \cite{Deng2025DecenProjRG} and DPRGT \cite{Wang2025DecenPGTRiem} study projected decentralized Riemannian gradient updates and decentralized proximal gradient tracking on Riemannian manifolds, respectively; together they provide our main projection based and tracking based baselines. Riemannian EXTRA (REXTRA) \cite{Wu2025RiemannianEC} further demonstrates that EXTRA type primal-dual correction can be extended from Euclidean space to compact submanifolds, but still through retraction or projection based updates.

Related works further considered quantization and compression \cite{Chen2025QuantizedRGT,Hu2024SingleStepCompression}, stochastic recursive momentum \cite{Deng2025DecenProjRSG}, conjugate gradient and natural gradient directions \cite{Chen2023DRCG,Hu2025DRNG}, local linear consensus behavior on the Stiefel manifold \cite{Chen2023LocalLinearStiefel}, manifold consensus formulations \cite{Tron2012RiemannianConsensus,Sarlette2009ConsensusManifolds,Hu2025LocalConsensus}, and decentralized retraction-free optimization on the Stiefel manifold \cite{Sun2024GlobalDecenRFStiefel}. \revise{However, these methods either rely on retractions/projections or do not provide an EXTRA type correction in the retraction-free Stiefel manifold regime, leaving open the question of whether one can obtain both retraction-free updates and primal-dual correction.}

\subsubsection{\texorpdfstring{\revise{Retraction-free optimization on the Stiefel manifold}}{Retraction-free optimization on the Stiefel manifold}}

\revise{A separate line of work studies Stiefel manifold optimization without performing explicit retractions at every iteration. For decentralized Stiefel manifold optimization, DESTINY \cite{Wang2022StiefelAL} is an early retraction-free method that uses an approximate augmented Lagrangian formulation to control Stiefel manifold feasibility through a penalty-based surrogate rather than explicit retractions. Building on the retraction-free perspective, DRFGT \cite{Sun2024GlobalDecenRFStiefel} further develops a decentralized gradient-tracking method with convergence guarantees for nonconvex Stiefel manifold problems. These works demonstrate that retraction-free updates can be made effective in decentralized Stiefel manifold optimization. However, existing retraction-free decentralized methods are still primal or gradient-tracking schemes. How to combine retraction-free Stiefel manifold feasibility control with a single-communication-step, EXTRA-type primal-dual correction mechanism remains largely unexplored.}

\revise{In the centralized setting, retraction-free and infeasible methods have also shown that one can reduce repeated QR or polar decompositions while still controlling feasibility drift \cite{ablin2022fast,gao2019parallelizable,Xiao2024DissolvingConstraints}. Several of these methods can be understood through approximate gradient mappings or penalty-type reformulations for Stiefel manifold optimization, which replace exact projection or retraction steps with cheaper ambient-space updates. 
This motivates the present work: by embedding a retraction-free Stiefel manifold update into an EXTRA recursion, we aim to bring the computational advantage of retraction-free methods to a single-communication-step primal-dual decentralized algorithm with provable convergence.}

\subsection{\texorpdfstring{\revise{Contribution}}{Contribution}}
\label{subsec:contribution}

This paper combines decentralized primal-dual correction, distributed optimization on the Stiefel manifold, and \revise{retraction-free Stiefel manifold feasibility control} in a single scheme. The key idea is to embed an approximate gradient mapping for \revise{Stiefel manifold optimization} into an EXTRA recursion.

Our contributions are summarized as follows.
\begin{itemize}
    \item We design and implement RF-EXTRA, a distributed retraction-free EXTRA algorithm that combines an approximate gradient mapping for \revise{Stiefel manifold optimization} with an EXTRA-based decentralized recursion on the Stiefel manifold, and we validate this design through a systematic numerical study on PCA and LRMC over the Stiefel manifold.
    \item On the theory side, we show that the joint-error variables $(\mathbf X_k-\overline{\mathbf X}_k,\mathbf s_k-\overline{\mathbf s}_k)$ remain in a controllable region and satisfy a contractive recursion under an equivalent norm. This allows us to establish a summed joint-error bound, compare the different terms in the descent analysis of the surrogate function, and finally derive an exact $\mathcal O(1/K)$ convergence rate under a constant stepsize. In contrast to existing retraction-based EXTRA-type methods such as REXTRA \cite{Wu2025RiemannianEC} and existing retraction-free schemes without EXTRA correction, our analysis must couple retraction-free feasibility control with an EXTRA-type inexact primal-dual recursion; this combination is nontrivial because of the manifold nonconvexity, especially under a constant stepsize, as also noted in \cite{Wu2025RiemannianEC}.
\end{itemize}

\subsection{\texorpdfstring{\revise{Organization}}{Organization}}
\label{subsec:organization}

\revise{The rest of this paper is organized as follows. Section~\ref{sec:method} presents RF-EXTRA and the approximate gradient mapping for Stiefel manifold optimization. Section~\ref{sec:theory} establishes the convergence theory and proves the exact $\mathcal O(1/K)$ convergence rate. Section~\ref{sec:experiments} reports numerical experiments on decentralized PCA and decentralized low-rank matrix completion. Section~\ref{sec:conclusion} concludes the paper.}

\subsection{Notation}
\label{subsec:notation}

Throughout the theoretical development, we use $n$ for the number of agents and $X\in\R^{d\times r}$ for the matrix variable on the Stiefel manifold. We write
\[
\mathbf X_k=[X_{1,k}^\top,\dots,X_{n,k}^\top]^\top,
\qquad
\mathbf s_k=[s_{1,k}^\top,\dots,s_{n,k}^\top]^\top
\]
for the stacked primal and auxiliary variables, and
\[
\bar X_k:=\frac1n\sum_{i=1}^n X_{i,k},
\qquad
\bar s_k:=\frac1n\sum_{i=1}^n s_{i,k}
\]
for their averages. The consensus projector is
\[
J:=\frac1n\mathbf 1_n\mathbf 1_n^\top\otimes I_d,
\]
and we use
\[
\overline{\mathbf X}_k:=(\mathbf 1_n\otimes I_d)\bar X_k,
\qquad
\overline{\mathbf s}_k:=(\mathbf 1_n\otimes I_d)\bar s_k.
\]
We also use the lifted mixing matrices
\[
\mathbf W:=W\otimes I_d,
\qquad
\mathbf V:=V\otimes I_d,
\]
so the stacked block recursions are written directly in terms of $\mathbf W$ and $\mathbf V$. The stacked Euclidean gradients are denoted by
\[
\nabla \mathbf f(\mathbf X):=[\nabla f_1(X_1)^\top,\dots,\nabla f_n(X_n)^\top]^\top.
\]
Finally, $\|\cdot\|_F$ denotes the Frobenius norm, and we use $[n]:=\{1,2,3,\cdots, n\}$.

\section{A Retraction-Free EXTRA Method for Decentralized Optimization}
\label{sec:method}

\revise{We consider the decentralized optimization problem on the Stiefel manifold introduced in Section~\ref{sec:introduction}, where each agent $i\in\{1,\dots,n\}$ maintains a local variable $X_{i,k}\in\R^{d\times r}$ and has access only to its local objective $f_i$. Under the consensus constraint $X_1=\cdots=X_n$, this problem reduces to minimizing the averaged objective}
\[
\revise{f(X):=\frac1n\sum_{i=1}^n f_i(X)}
\]
\revise{over $\St(d,r)$. Let $W\in\R^{n\times n}$ be a symmetric doubly stochastic mixing matrix associated with a connected undirected graph, and let}
\[
V=\theta I_n+(1-\theta)W,
\qquad
\theta\in(0,\frac12].
\]
We use the lifted mixing matrices $\mathbf W$ and $\mathbf V$ introduced in the notation section. Our construction starts from the Euclidean correction philosophy of EXTRA and then replaces the Euclidean gradient with a surrogate compatible with the geometry of the Stiefel manifold.

For unconstrained consensus optimization, EXTRA corrects the steady state bias of decentralized gradient descent by coupling the primal recursion with a memory term. An equivalent two-state representation of the Euclidean mechanism can be written as
\begin{equation}
\label{eq:extra-state-primal}
\mathbf X_{k+1}=\mathbf W\mathbf X_k+\mathbf s_k,
\end{equation}
\begin{equation}
\label{eq:extra-state-aux-generic}
\mathbf s_{k+1}=\mathbf s_k+(\mathbf W-\mathbf V)\mathbf X_k-\alpha\bigl(\nabla \mathbf f(\mathbf X_{k+1})-\nabla \mathbf f(\mathbf X_k)\bigr),
\end{equation}
where $\nabla \mathbf f(\mathbf X_k)$ is defined in the notation section. This representation makes the role of the auxiliary state explicit: the matrix $\mathbf W-\mathbf V$ handles the consensus correction, while the gradient difference drives descent without introducing a persistent bias under heterogeneity \cite{shi2015extra,yuan2018exact}. Our goal is to retain this correction architecture in the setting of the Stiefel manifold.

\revise{To transfer the approximate augmented Lagrangian idea of DESTINY into the EXTRA framework, we next introduce the global surrogate functions associated with the averaged objective $f$, namely}
\[
\revise{b(X):=\frac14\|X^\top X-I_r\|_F^2,}
\qquad
\revise{g(X):=\frac32 f(X)-\frac12 f(XX^\top X),}
\qquad
\revise{h(X):=g(X)+\beta b(X).}
\]
Here $h$ plays the role of an approximate augmented Lagrangian in ambient space. Its gradient is
\[
\nabla h(X)=\nabla g(X)+\beta X(X^\top X-I_r),
\]
where
\[
\nabla g(X)=\frac32\nabla f(X)-\frac12\nabla f(XX^\top X)X^\top X-X\,\sym\!\bigl(X^\top\nabla f(XX^\top X)\bigr).
\]
The important point is that the exact gradient $\nabla g(X)$ uses the gradient of $f$ at two different points, namely $X$ and $XX^\top X$. Thus, a direct implementation of $\nabla h(X)$ would require two gradient evaluations within one iteration. To avoid this extra cost, we replace $\nabla g(X)$ by the approximate gradient mapping
\[
G(X)=\nabla f(XX^\top X)\frac{3I_r-X^\top X}{2}-X\,\sym\!\bigl(X^\top\nabla f(XX^\top X)\bigr),
\]
and define
\[
H(X)=G(X)+\beta X(X^\top X-I_r).
\]
With this choice, both $G(X)$ and $H(X)$ are built from a single gradient evaluation at the projected point $XX^\top X$. This is precisely why we use the present $G$ and $H$: each iteration only needs one evaluation of $\nabla f(XX^\top X)$ rather than separate evaluations at $X$ and $XX^\top X$.

We now introduce the corresponding local maps used by the decentralized algorithm. For each agent, define
\[
G_i(X)=\nabla f_i(XX^\top X)\frac{3I_r-X^\top X}{2}-X\,\sym\!\bigl(X^\top\nabla f_i(XX^\top X)\bigr),
\qquad
H_i(X)=G_i(X)+\beta X(X^\top X-I_r).
\]
The first term provides an ambient descent surrogate tailored to the geometry of the Stiefel manifold, while the penalty term controls orthogonality violation. This construction lets the geometric information enter through the local map $H_i$ without changing the linear communication structure of EXTRA. For the averaged analysis, we use the corresponding global maps $G$ and $H$ defined above, together with
\[
\mathbf H_k:=[H_1(X_{1,k})^\top,\dots,H_n(X_{n,k})^\top]^\top,
\qquad
\bar H_k:=\frac1n\sum_{i=1}^n H_i(X_{i,k}).
\]

The resulting algorithm, called RF-EXTRA, is stated in Algorithm~\ref{alg:rfextra}. The initialization follows the EXTRA mechanism and uses the local map $H_i$ evaluated at the initial point.

\begin{algorithm}[t]
\caption{RF-EXTRA}
\label{alg:rfextra}
\begin{algorithmic}[1]
\Require Mixing matrix $W$, correction matrix $V=\theta I_n+(1-\theta)W$ with $\theta\in(0,1/2]$, stepsize $\alpha>0$, penalty parameter $\beta>0$, initial matrices $X_{i,0}\in\R^{d\times r}$ for $i=1,\dots,n$
\For{each agent $i=1,\dots,n$}
    \State Compute $H_i(X_{i,0})$
    \State Set $s_{i,0}=-\alpha H_i(X_{i,0})$
\EndFor
\For{$k=0,1,2,\dots$}
    \For{each agent $i=1,\dots,n$ in parallel}
        \State Exchange $X_{j,k}$ with neighbors $j$
        \State Update the primal variable:
        \[
        X_{i,k+1}=\sum_{j=1}^n w_{ij}X_{j,k}+s_{i,k}
        \]
        \State Compute $H_i(X_{i,k+1})$
        \State Update the auxiliary variable:
        \[
        s_{i,k+1}=s_{i,k}+\sum_{j=1}^n (w_{ij}-v_{ij})X_{j,k}-\alpha\bigl(H_i(X_{i,k+1})-H_i(X_{i,k})\bigr)
        \]
    \EndFor
\EndFor
\end{algorithmic}
\end{algorithm}

In stacked form, Algorithm~\ref{alg:rfextra} reads
\begin{equation}
\label{eq:rfextra-primal}
\mathbf X_{k+1}=\mathbf W\mathbf X_k+\mathbf s_k,
\end{equation}
\begin{equation}
\label{eq:rfextra-aux}
\mathbf s_{k+1}=\mathbf s_k+(\mathbf W-\mathbf V)\mathbf X_k-\alpha(\mathbf H_{k+1}-\mathbf H_k).
\end{equation}
Averaging \eqref{eq:rfextra-primal}--\eqref{eq:rfextra-aux} and using the doubly stochasticity of $W$ and $V$ yields
\[
\bar s_k=-\alpha\bar H_k,
\qquad
\bar X_{k+1}=\bar X_k-\alpha\bar H_k.
\]
This averaged recursion is the key link between RF-EXTRA and the descent analysis developed in Section~\ref{sec:theory}.

RF-EXTRA has two main advantages. First, it preserves the EXTRA mechanism, so the communication layer remains linear and fully decentralized. Second, it avoids explicit retractions at every iteration. Each step requires one neighbor mixing update for $\mathbf X_k$, one evaluation of the local surrogate map $H_i$, and one auxiliary correction update for $\mathbf s_k$. Consequently, the method replaces repeated orthogonalization with a cheaper ambient-space computation while still maintaining effective control of feasibility and consensus through the penalty term and the correction state.

\section{Convergence Analysis}
\label{sec:theory}

This section establishes convergence of the averaged RF-EXTRA iterate. The analysis combines two ingredients: descent properties inherited from the centralized surrogate analysis on the Stiefel manifold and a contractive recursion for the joint error. 

\subsection{Standing assumptions}
\label{subsec:standing-assumptions}

Before stating the assumptions, we collect the constants used throughout the analysis. Let
\[
b(X):=\frac14\|X^\top X-I_r\|_F^2,
\qquad
\mathcal R:=\{X\in\mathbb R^{d\times r}:\|X^\top X-I_r\|_F\le 1/6\},
\]
and define the explicit bounded set
\[
\mathcal B:=\{X\in\mathbb R^{d\times r}:\|X\|_F\le \sqrt{7r/6}+1\}.
\]
We use $L_f$ for a uniform Lipschitz constant of the local gradients on $\mathcal B_f:=\mathcal B\cup\{XX^\top X:X\in\mathcal B\}$,
namely
\[
L_f:=\sup_{i\in [n]}\sup_{X\neq Y\in\mathcal B_f}
\frac{\|\nabla f_i(X)-\nabla f_i(Y)\|_F}{\|X-Y\|_F},
\]
and we define
\[
\revise{L_g:=\sup_{i\in [n]}\sup_{X\neq Y\in\mathcal B}\frac{\|G_i(X)-G_i(Y)\|_F}{\|X-Y\|_F},}
\qquad
L_b:=\sup_{X\neq Y\in\mathcal B}\frac{\|\nabla b(X)-\nabla b(Y)\|_F}{\|X-Y\|_F},
\]
\[
M_g:=\max_{i\in [n]}\sup_{X\in\mathcal R}\|G_i(X)\|_F,
\qquad
C_0:=\max_{i\in [n]}\sup_{X\in\mathcal R}\|\nabla f_i(XX^\top X)\|_F,
\]
and the shorthand
\[
L_H:=L_g+\beta L_b,
\qquad
\revise{L_h:=\sup_{X\neq Y\in\mathcal B}\frac{\|\nabla h(X)-\nabla h(Y)\|_F}{\|X-Y\|_F},}
\qquad
M_H:=\max_{i\in [n]}\sup_{X\in\mathcal R}\|H_i(X)\|_F.
\]
Since $f$, $G$, and $H$ are the averages of $f_i$, $G_i$, and $H_i$, respectively, these constants also bound the corresponding averaged quantities on $\mathcal R$.

We collect the assumptions used in the convergence analysis.

\begin{assumption}[Smooth local models]
\label{ass:smooth}
\revise{Each $\nabla f_i$ is Lipschitz continuous, and the constants $L_f$, $L_g$, $L_b$, $M_g$, $C_0$, $L_H$, $L_h$, and $M_H$ defined above are finite on the indicated sets.}
\end{assumption}

\begin{assumption}[Network]
\label{ass:network}
The mixing matrix $W$ is symmetric and doubly stochastic, the communication graph is connected, and the second largest singular value satisfies $\sigma_2(W)<1$.
\end{assumption}

\subsection{\texorpdfstring{\revise{Averaged neighborhood preservation}}{Averaged neighborhood preservation}}
\label{subsec:averaged-sequence}

\revise{Similar to \cite[Lemma~2]{Wang2022StiefelAL}, the next lemma records the two properties imported from the centralized analysis: coercivity of the surrogate map on the neighborhood of interest and one-step propagation of the averaged iterate inside that neighborhood.}

\begin{lemma}[\revise{Averaged neighborhood preservation}]
\label{lem:destiny-ingredients}
If
\[
\beta\ge \frac{6+21C_0}{5},
\]
then, for every $X\in\mathcal R$,
\begin{equation}
\label{eq:destiny-coercivity-rfextra}
\|H(X)\|_F^2\ge \|G(X)\|_F^2+\beta\|X^\top X-I_r\|_F^2.
\end{equation}
Moreover, if
\[
\bar X_k\in\mathcal R,
\qquad
\|\mathbf X_k-\overline{\mathbf X}_k\|_F\le \min\left\{1,\frac{\sqrt n}{L_H}\right\},
\qquad
\alpha\le \min\left\{\frac{1}{4L_H},\,\frac{1}{4\beta L_b},\,\frac{1}{4\beta}\right\},
\]
and
\[
\beta\ge 12\sqrt2\left(M_g+1\right),
\]
then
\[
\bar X_{k+1}\in\mathcal R.
\]
\end{lemma}

\begin{proof}
The coercivity estimate \eqref{eq:destiny-coercivity-rfextra} is exactly the statement of DESTINY Lemma~2 for the same pointwise map
\[
H(X)=G(X)+\beta X(X^\top X-I_r)
\]
on the same region $\mathcal R$; hence it may be cited directly.

For the second claim, let
\[
Q(X):=X^\top X-I_r,
\qquad
b(X):=\frac14\|Q(X)\|_F^2.
\]
Then $\nabla b(X)=XQ(X)$. As before, if
\[
J:=\frac1n\mathbf 1_n\mathbf 1_n^\top\otimes I_d,
\]
then $J\mathbf W=J$, $J\mathbf V=J$, and $J(\mathbf W-\mathbf V)=0$. Applying $J$ to the auxiliary recursion gives
\[
J\mathbf s_{k+1}=J\mathbf s_k-\alpha J(\mathbf H_{k+1}-\mathbf H_k).
\]
Because $s_{i,0}=-\alpha H_i(X_{i,0})$, averaging yields $\bar s_0=-\alpha\bar H_0$, and induction gives
\[
\bar s_k=-\alpha\bar H_k\qquad\text{for all }k\ge0.
\]
Applying $J$ to the primal recursion,
\[
\mathbf X_{k+1}=\mathbf W\mathbf X_k+\mathbf s_k,
\]
and using $J\mathbf W=J$, we obtain
\[
\bar X_{k+1}=\bar X_k+\bar s_k=\bar X_k-\alpha\bar H_k.
\]
Now define
\[
\xi_k:=H(\bar X_k)-\bar H_k-G(\bar X_k).
\]
Since $H(\bar X_k)=G(\bar X_k)+\beta\nabla b(\bar X_k)$, the averaged recursion becomes
\begin{equation}
\label{eq:avg-penalty-step-rfextra}
\bar X_{k+1}=\bar X_k-\alpha\beta\nabla b(\bar X_k)+\alpha\xi_k.
\end{equation}
Moreover,
\[
\revise{\|X_{i,k}-\bar X_k\|_F\le \|\mathbf X_k-\overline{\mathbf X}_k\|_F\le 1,\qquad i=1,\dots,n.}
\]
\revise{Since $\bar X_k\in\mathcal R$, we have $\|\bar X_k^\top\bar X_k-I_r\|_F\le 1/6$, and hence}
\[
\revise{\|\bar X_k\|_F^2=\operatorname{tr}(\bar X_k^\top\bar X_k)
\le \operatorname{tr}(I_r)+\sqrt r\,\|\bar X_k^\top\bar X_k-I_r\|_F
\le r+\frac{\sqrt r}{6}
\le \frac{7r}{6}.}
\]
\revise{Therefore $\|\bar X_k\|_F\le \sqrt{7r/6}$, hence $X_{i,k}\in\mathcal B$ for all $i$. Therefore $H_i$ is $L_H$-Lipschitz at the points used below, and}
\begin{equation}
\label{xikestimation}
\|\xi_k\|_F
\le \|H(\bar X_k)-\bar H_k\|_F+\|G(\bar X_k)\|_F
\le \frac{L_H}{\sqrt n}\|\mathbf X_k-\overline{\mathbf X}_k\|_F+M_g
\le 1+M_g.
\end{equation}
\revise{Because $\bar X_k\in\mathcal R$, we have $\|Q(\bar X_k)\|_F\le 1/6$, so every eigenvalue of $\bar X_k^\top\bar X_k$ lies in $[5/6,7/6]$.}
\revise{We now estimate the feasibility of $\bar X_{k+1}$ directly through the singular values of $\bar X_k$. Let}
\[
\revise{
t:=\alpha\beta,\qquad
\bar X_k=U\Sigma V^\top,\qquad
q_j:=\sigma_j^2-1,}
\]
\revise{where $\Sigma=\operatorname{diag}(\sigma_j)$. Since $\bar X_k\in\mathcal R$, we have $|q_j|\le 1/6$. Define the penalty-only point}
\[
\revise{
Z_k:=\bar X_k-t\nabla b(\bar X_k).
}
\]
\revise{Using $\nabla b(\bar X_k)=U\Sigma(\Sigma^2-I_r)V^\top$, we have}
\[
\revise{
Z_k
=U\operatorname{diag}\bigl(\sigma_j(1-tq_j)\bigr)V^\top .
}
\]
\revise{Therefore}
\begin{equation}
\label{eq:penalty-feasibility-recursion-rfextra}
\revise{
\begin{aligned}
\|Z_k^\top Z_k-I_r\|_F^2
&=\sum_j\left((1+q_j)(1-tq_j)^2-1\right)^2\\
&=\sum_j q_j^2\left(1-2t(1+q_j)+t^2q_j(1+q_j)\right)^2\\
&\le \left(1-\frac53t\right)^2\sum_j q_j^2\\
&=\left(1-\frac53t\right)^2\|\bar X_k^\top\bar X_k-I_r\|_F^2 .
\end{aligned}
}
\end{equation}
\revise{Indeed, the scalar factor in the second line is bounded by $1-5t/3$ for $q_j\in[-1/6,1/6]$ and $t\le1/4$: if $q_j\ge0$, then
$1-2t(1+q_j)+t^2q_j(1+q_j)\le 1-2t\le 1-5t/3$; if $q_j<0$, then
$1-2t(1+q_j)+t^2q_j(1+q_j)\le1-2t(5/6)=1-5t/3$.}

\revise{It remains to include the perturbation $\alpha\xi_k$ in \eqref{eq:avg-penalty-step-rfextra}. Since the singular values of $Z_k$ are $\sigma_j|1-tq_j|$, the bounds $|q_j|\le1/6$ and $t\le1/4$ imply $\|Z_k\|_2\le \sqrt{7/6}$. Hence, using $\bar X_{k+1}=Z_k+\alpha\xi_k$,}
\[
\revise{
\begin{aligned}
\|\bar X_{k+1}^\top\bar X_{k+1}-I_r\|_F
&\le \|Z_k^\top Z_k-I_r\|_F
   +2\alpha\|Z_k\|_2\|\xi_k\|_F+\alpha^2\|\xi_k\|_F^2\\
&\le \left(1-\frac53t\right)\|\bar X_k^\top\bar X_k-I_r\|_F
   +2t\sqrt{\frac76}\frac{\|\xi_k\|_F}{\beta}
   +t^2\left(\frac{\|\xi_k\|_F}{\beta}\right)^2 .
\end{aligned}
}
\]

\revise{Combine the assumed $\beta$ condition and \eqref{xikestimation}, we have}
\[
\revise{
\frac{\|\xi_k\|_F}{\beta}
\le
\frac{1}{12\sqrt2}.}
\]
\revise{Since $\|\bar X_k^\top\bar X_k-I_r\|_F\le1/6$, $t\le1/4$, and}
\[
\revise{
2\sqrt{\frac76}\frac{1}{12\sqrt2}
+\frac{t}{288}
\le \frac16+\frac1{1152}
\le \frac{5}{18},}
\]
\revise{we obtain}
\[
\revise{
\|\bar X_{k+1}^\top\bar X_{k+1}-I_r\|_F
\le \left(1-\frac53t\right)\frac16+\frac{5t}{18}
=\frac16 .}
\]
\revise{Hence $\bar X_{k+1}\in\mathcal R$, as claimed.}
\end{proof}

\subsection{Joint-error recursion in an equivalent norm}
\label{subsec:joint-error}

We next control the disagreement-correction subsystem associated with RF-EXTRA. The key point is that one should not work with the Frobenius operator norm of the linear transition matrix. We therefore begin with a standard equivalent-norm lemma.

\begin{lemma}[\cite{horn2012matrix}]
\label{lem:radius}
For any given $\varepsilon > 0$, there exists an operator norm $\|\cdot\|_S$ (dependent on $A$ and $\varepsilon$) such that $\|A\|_S \leq \rho(A) + \varepsilon$. Consequently, if $\rho(A) < 1$, then there exists an operator norm $\|\cdot\|_S$ for which $\|A\|_S < 1$.
\end{lemma}

Under Assumption~\ref{ass:network}, define
\[
P:=\begin{bmatrix}
\mathbf W-J & I\\
\mathbf W-\mathbf V & I-J
\end{bmatrix}.
\]
\revise{We define the joint error by}
\[
\revise{
\mathbf p_k:=
\begin{bmatrix}
\mathbf X_k-\overline{\mathbf X}_k\\[1mm]
\mathbf s_k-\overline{\mathbf s}_k
\end{bmatrix}.}
\]
Using the standard state-space representation of EXTRA \cite{shi2015extra,yuan2018exact}, the RF-EXTRA disagreement state satisfies
\begin{equation}
\label{eq:p-recursion-intro}
\mathbf p_{k+1}=P\mathbf p_k+\alpha\,\boldsymbol\delta_k,
\qquad
\boldsymbol\delta_k:=\begin{bmatrix}
0\\ (J-I)(\mathbf H_{k+1}-\mathbf H_k)
\end{bmatrix}.
\end{equation}
By \cite{qin2025convergence}, the condition $\sigma_2(W)<1$ implies $\rho(P)<1$. Applying Lemma~\ref{lem:radius} to $A=P$, we may therefore fix an operator norm $\|\cdot\|_S$ and constants $c_1,c_2>0$ such that
\[
\|P\|_S=: \sigma<1,
\qquad
c_1\|A\|_F\le \|A\|_S\le c_2\|A\|_F
\]
for all stacked variables $A$ in the joint-error space.

\revise{Lemma~\ref{lem:joint-error-recursion} below establishes that this joint error is contractive under the equivalent norm, up to the perturbation induced by the averaged Stiefel manifold update.}

\begin{lemma}[Joint-error recursion in an equivalent norm]
\label{lem:joint-error-recursion}
Recall the stacked quantities $\mathbf p_k$, $\mathbf H_k$, and $\sigma$ introduced earlier. Let $c_1,c_2>0$ satisfy
\[
c_1\|A\|_F\le \|A\|_S\le c_2\|A\|_F.
\]
\revise{Let}
\[
\revise{
a:=\frac{1+\sigma}{2},
\qquad
B_H:=c_2L_H\sqrt n.}
\]
\revise{Assume that, for some $\delta_S>0$,}
\begin{align*}
\revise{\bar X_k\in\mathcal R,\qquad
\|\mathbf p_k\|_S\le \delta_S,\qquad
\delta_S\le c_1\min\left\{\frac{1}{1+\sigma_2(W)},\frac{\sqrt n}{L_H}\right\},}\\
\revise{\alpha\le \min\left\{\frac1{4L_H},\frac1{4\beta L_b},\frac1{4\beta},\frac{c_1(1-\sigma)}{8c_2L_H},1\right\},\qquad
\beta\ge 12\sqrt2\left(M_g+1\right).}
\end{align*}
\revise{Then $\bar X_{k+1}\in\mathcal R$, $X_{i,k}, X_{i,k+1}\in\mathcal B$ for all $i$ and}
\begin{equation}
\label{eq:joint-recursion-final}
\|\mathbf p_{k+1}\|_S\le a\|\mathbf p_k\|_S+\alpha^2 B_H\|H(\bar X_k)\|_F.
\end{equation}
\end{lemma}

\begin{proof}
Let
\[
\widetilde{\mathbf X}_k:=\mathbf X_k-\overline{\mathbf X}_k,
\qquad
\widetilde{\mathbf s}_k:=\mathbf s_k-\overline{\mathbf s}_k.
\]
From \eqref{eq:p-recursion-intro}, we have
\begin{equation}
\label{eq:p-basic-bound}
\|\mathbf p_{k+1}\|_S\le \sigma\|\mathbf p_k\|_S+\alpha\|\boldsymbol\delta_k\|_S.
\end{equation}
\revise{By norm equivalence and the bound $\|\mathbf p_k\|_S\le \delta_S$,}
\[
\revise{\|\widetilde{\mathbf X}_k\|_F\le \|\mathbf p_k\|_F\le \frac{\delta_S}{c_1},
\qquad
\|\widetilde{\mathbf s}_k\|_F\le \|\mathbf p_k\|_F\le \frac{\delta_S}{c_1}.}
\]
\revise{Since $\bar X_k\in\mathcal R$, we have $\|\bar X_k^\top\bar X_k-I_r\|_F\le 1/6$, and hence}
\[
\revise{\|\bar X_k\|_F^2=\operatorname{tr}(\bar X_k^\top\bar X_k)
\le \operatorname{tr}(I_r)+\sqrt r\,\|\bar X_k^\top\bar X_k-I_r\|_F
\le r+\frac{\sqrt r}{6}
\le \frac{7r}{6}.}
\]
\revise{Therefore $\|\bar X_k\|_F\le \sqrt{7r/6}$. Together with $\delta_S/c_1\le 1/(1+\sigma_2(W))\le 1$, this yields}
\[
\revise{\|X_{i,k}\|_F\le \|\bar X_k\|_F+\|X_{i,k}-\bar X_k\|_F
\le \sqrt{7r/6}+\|\widetilde{\mathbf X}_k\|_F
\le \sqrt{7r/6}+1,\qquad i=1,\dots,n,}
\]
\revise{hence $X_{i,k}\in\mathcal B$ for all $i$.}

\revise{The assumptions imply $\|\mathbf X_k-\overline{\mathbf X}_k\|_F\le\delta_S/c_1\le \min\{1,\sqrt n/L_H\}$, and hence all conditions in the second part of Lemma~\ref{lem:destiny-ingredients} are satisfied. Therefore}
\[
\revise{\bar X_{k+1}\in\mathcal R.}
\]
\revise{Moreover, subtracting the average from the primal recursion yields}
\[
\revise{\widetilde{\mathbf X}_{k+1}=(\mathbf W-J)\widetilde{\mathbf X}_k+\widetilde{\mathbf s}_k.}
\]
\revise{Therefore}
\[
\revise{\|\widetilde{\mathbf X}_{k+1}\|_F
\le \sigma_2(W)\|\widetilde{\mathbf X}_k\|_F+\|\widetilde{\mathbf s}_k\|_F
\le (1+\sigma_2(W))\frac{\delta_S}{c_1}
\le 1.}
\]
\revise{Using again $\bar X_{k+1}\in\mathcal R$, we obtain $\|\bar X_{k+1}\|_F\le \sqrt{7r/6}$ and hence}
\[
\revise{\|X_{i,k+1}\|_F\le \|\bar X_{k+1}\|_F+\|X_{i,k+1}-\bar X_{k+1}\|_F
\le \sqrt{7r/6}+\|\widetilde{\mathbf X}_{k+1}\|_F
\le \sqrt{7r/6}+1,\qquad i=1,\dots,n,}
\]
\revise{so $X_{i,k+1}\in\mathcal B$ for all $i$. Thus every point at which $H_i$ is evaluated in the present step lies in $\mathcal B$.}
Since $J(\mathbf W-\mathbf V)=0$ and $\overline{\mathbf s}_k=-\alpha(\mathbf 1_n\otimes I_d)\bar H_k$, the primal increment satisfies
\[
\mathbf X_{k+1}-\mathbf X_k=(\mathbf W-I)\widetilde{\mathbf X}_k+\widetilde{\mathbf s}_k-\alpha\overline{\mathbf H}_k,
\]
where $\bar H_k:=n^{-1}\sum_{i=1}^n H_i(X_{i,k})$ and $\overline{\mathbf H}_k:=(\mathbf 1_n\otimes I_d)\bar H_k$. Therefore,
\begin{equation}
\label{eq:X-inc-bound}
\|\mathbf X_{k+1}-\mathbf X_k\|_F
\le 2\|\widetilde{\mathbf X}_k\|_F+\|\widetilde{\mathbf s}_k\|_F+\alpha\|\overline{\mathbf H}_k\|_F.
\end{equation}
\revise{Because $X_{i,k},X_{i,k+1}\in\mathcal B$ for every $i$, each $H_i$ is $L_H$-Lipschitz along this step, and therefore}
\[
\|H(\bar X_k)-\bar H_k\|_F
\le \frac{L_H}{n}\sum_{i=1}^n\|\bar X_k-X_{i,k}\|_F
\le \frac{L_H}{\sqrt n}\|\mathbf X_k-\overline{\mathbf X}_k\|_F,
\]
which implies
\begin{equation}
\label{eq:barH-bound-final}
\|\overline{\mathbf H}_k\|_F
\le L_H\|\widetilde{\mathbf X}_k\|_F+\sqrt n\,\|H(\bar X_k)\|_F.
\end{equation}
Substituting \eqref{eq:barH-bound-final} into \eqref{eq:X-inc-bound} and using $\alpha L_H\le1$ give
\[
\|\mathbf X_{k+1}-\mathbf X_k\|_F
\le 3\|\widetilde{\mathbf X}_k\|_F+\|\widetilde{\mathbf s}_k\|_F+\alpha\sqrt n\,\|H(\bar X_k)\|_F
\le 4\|\mathbf p_k\|_F+\alpha\sqrt n\,\|H(\bar X_k)\|_F.
\]
\revise{Therefore we have}
\[
\|\mathbf H_{k+1}-\mathbf H_k\|_F\le L_H\|\mathbf X_{k+1}-\mathbf X_k\|_F
\le 4L_H\|\mathbf p_k\|_F+\alpha L_H\sqrt n\,\|H(\bar X_k)\|_F.
\]
Hence
\[
\|\boldsymbol\delta_k\|_S
\le c_2\|\mathbf H_{k+1}-\mathbf H_k\|_F
\le \frac{4c_2L_H}{c_1}\|\mathbf p_k\|_S+\alpha B_H\|H(\bar X_k)\|_F.
\]
Combining this with \eqref{eq:p-basic-bound} yields
\[
\|\mathbf p_{k+1}\|_S
\le \left(\sigma+\frac{4c_2L_H}{c_1}\alpha\right)\|\mathbf p_k\|_S+\alpha^2B_H\|H(\bar X_k)\|_F.
\]
The restriction $\alpha\le c_1(1-\sigma)/(8c_2L_H)$ gives
\[
\sigma+\frac{4c_2L_H}{c_1}\alpha\le \frac{1+\sigma}{2}=a,
\]
which proves \eqref{eq:joint-recursion-final}.
\end{proof}

\paragraph{Remark.}
A direct Frobenius-norm contraction argument is unavailable for the RF-EXTRA transition matrix $P$. Indeed, if $u\neq0$ satisfies $Ju=0$ and $z=[0;u]^\top$, then $Pz=[u;u]^\top$, so $\|Pz\|_F=\sqrt2\,\|u\|_F>\|z\|_F$. Thus any contraction inequality directly under $\|\cdot\|_F$ is infeasible. This is why the introduction of the $S$-norm is crucial.

\subsection{Neighborhood propagation and boundedness}
\label{subsec:boundedness}

Once the joint-error recursion is available, we propagate the invariant neighborhood.

\begin{lemma}[One-step propagation of the joint neighborhood]
\label{lem:star_inv}
Let $\delta_S>0$, and define
\[
\mathcal N_S(\delta_S):=\{(\mathbf X,\mathbf s):\|\mathbf p\|_S\le \delta_S\}.
\]
Let
\[
B_H:=c_2L_H\sqrt n,
\qquad
a:=\frac{1+\sigma}{2}.
\]
Assume the hypotheses of Lemma~\ref{lem:joint-error-recursion}. If, for some $k\ge0$,
\[
\bar X_k\in\mathcal R,
\qquad
(\mathbf X_k,\mathbf s_k)\in\mathcal N_S(\delta_S),
\qquad
\alpha\le \sqrt{\frac{(1-a)\delta_S}{B_HM_H}},
\]
then
\[
\revise{\bar X_{k+1}\in\mathcal R,
\qquad}
(\mathbf X_{k+1},\mathbf s_{k+1})\in\mathcal N_S(\delta_S).
\]
\end{lemma}

\begin{proof}
\revise{By Lemma~\ref{lem:joint-error-recursion}, we have $\bar X_{k+1}\in\mathcal R$ and}
\[
\|\mathbf p_{k+1}\|_S\le a\|\mathbf p_k\|_S+\alpha^2B_H\|H(\bar X_k)\|_F.
\]
Since $\bar X_k\in\mathcal R$ and $H(\bar X_k)=n^{-1}\sum_{i=1}^n H_i(\bar X_k)$, the definition of $M_H$ implies
\[
\|H(\bar X_k)\|_F\le M_H.
\]
Using $\|\mathbf p_k\|_S\le \delta_S$, we obtain
\[
\|\mathbf p_{k+1}\|_S\le a\delta_S+\alpha^2B_HM_H.
\]
The stepsize restriction yields
\[
\alpha^2B_HM_H\le (1-a)\delta_S,
\]
hence
\[
\|\mathbf p_{k+1}\|_S\le a\delta_S+(1-a)\delta_S=\delta_S.
\]
Hence $(\mathbf X_{k+1},\mathbf s_{k+1})\in\mathcal N_S(\delta_S)$.
\end{proof}

The recursive bound above also yields a summed estimate for the joint error, which will be used in the final descent argument.

\begin{lemma}[Summed joint-error bound]
\label{lem:joint-error-sum}
\revise{Let $\delta_S>0$ satisfy}
\[
\revise{\delta_S\le c_1\min\left\{\frac{1}{1+\sigma_2(W)},\frac{\sqrt n}{L_H}\right\}.}
\]
\revise{Assume $\bar X_0\in\mathcal R$, $(\mathbf X_0,\mathbf s_0)\in\mathcal N_S(\delta_S)$,}
\[
\revise{\alpha\le \min\left\{\frac1{4L_H},\frac1{4\beta L_b},\frac1{4\beta},\frac{c_1(1-\sigma)}{8c_2L_H},\sqrt{\frac{(1-a)\delta_S}{B_HM_H}},1\right\},}
\]
\revise{and}
\[
\revise{\beta\ge 12\sqrt2\left(M_g+1\right).}
\]
\revise{Then, for every integer $K\ge0$,}
\begin{equation}
\label{eq:joint-sum-final}
\sum_{k=0}^{K}\Bigl(\|\mathbf X_{k+1}-\overline{\mathbf X}_{k+1}\|_F^2+\|\mathbf s_{k+1}-\overline{\mathbf s}_{k+1}\|_F^2\Bigr)
\le C_{P,0}+\alpha^2 C_{P,1}\sum_{k=0}^{K}\|H(\bar X_k)\|_F^2,
\end{equation}
where
\[
C_{P,0}:=\frac{2a^2}{c_1^2(1-a^2)}\|\mathbf p_0\|_S^2,
\qquad
C_{P,1}:=\frac{2n c_2^2L_H^2}{c_1^2(1-a)^2}.
\]
\end{lemma}

\begin{proof}
\revise{Repeated application of Lemma~\ref{lem:star_inv} gives $\bar X_k\in\mathcal R$ and $(\mathbf X_k,\mathbf s_k)\in\mathcal N_S(\delta_S)$ for all $k\ge0$. Hence Lemma~\ref{lem:joint-error-recursion} applies at every iteration.}
Unrolling \eqref{eq:joint-recursion-final} gives
\[
\|\mathbf p_{k+1}\|_S\le a^{k+1}\|\mathbf p_0\|_S+\alpha^2B_H\sum_{t=0}^k a^{k-t}\|H(\bar X_t)\|_F.
\]
Set $h_t:=\|H(\bar X_t)\|_F$. Using $(u+v)^2\le 2u^2+2v^2$, we obtain
\[
\|\mathbf p_{k+1}\|_S^2
\le 2a^{2k+2}\|\mathbf p_0\|_S^2
+2\alpha^4B_H^2\left(\sum_{t=0}^k a^{k-t}h_t\right)^2.
\]
Summing over $k=0,\dots,K$ yields
\[
\sum_{k=0}^{K}\|\mathbf p_{k+1}\|_S^2
\le 2\|\mathbf p_0\|_S^2\sum_{k=0}^{K}a^{2k+2}
+2\alpha^4B_H^2\sum_{k=0}^{K}\left(\sum_{t=0}^k a^{k-t}h_t\right)^2.
\]
The geometric-series bound gives
\[
\sum_{k=0}^{K}a^{2k+2}\le \frac{a^2}{1-a^2}.
\]
For the second term, Cauchy--Schwarz implies
\[
\left(\sum_{t=0}^k a^{k-t}h_t\right)^2
\le \left(\sum_{t=0}^k a^{k-t}\right)\left(\sum_{t=0}^k a^{k-t}h_t^2\right)
\le \frac{1}{1-a}\sum_{t=0}^k a^{k-t}h_t^2.
\]
Summing over $k$ and exchanging the order of summation,
\[
\sum_{k=0}^{K}\left(\sum_{t=0}^k a^{k-t}h_t\right)^2
\le \frac{1}{1-a}\sum_{t=0}^{K}h_t^2\sum_{k=t}^{K}a^{k-t}
\le \frac{1}{(1-a)^2}\sum_{t=0}^{K}h_t^2.
\]
Therefore,
\[
\sum_{k=0}^{K}\|\mathbf p_{k+1}\|_S^2
\le \frac{2a^2}{1-a^2}\|\mathbf p_0\|_S^2+
\frac{2\alpha^4B_H^2}{(1-a)^2}\sum_{k=0}^{K}\|H(\bar X_k)\|_F^2.
\]
Finally, norm equivalence gives
\[
\|\mathbf X_{k+1}-\overline{\mathbf X}_{k+1}\|_F^2+\|\mathbf s_{k+1}-\overline{\mathbf s}_{k+1}\|_F^2
\le \frac1{c_1^2}\|\mathbf p_{k+1}\|_S^2,
\]
and, since $\alpha\le1$ and $B_H=c_2L_H\sqrt n$,
\[
\sum_{k=0}^{K}\Bigl(\|\mathbf X_{k+1}-\overline{\mathbf X}_{k+1}\|_F^2+\|\mathbf s_{k+1}-\overline{\mathbf s}_{k+1}\|_F^2\Bigr)
\le \frac{2a^2}{c_1^2(1-a^2)}\|\mathbf p_0\|_S^2
+\alpha^2\frac{2n c_2^2L_H^2}{c_1^2(1-a)^2}\sum_{k=0}^{K}\|H(\bar X_k)\|_F^2.
\]
This is exactly \eqref{eq:joint-sum-final}.
\end{proof}

\subsection{Exact O(1/K) convergence rate}
\label{subsec:rate}

We now combine averaged descent with the summed joint-error bound. Summing the descent inequality along the averaged trajectory and comparing the different terms through Lemma~\ref{lem:joint-error-sum} yields an exact $\mathcal O(1/K)$ convergence rate. 

\begin{lemma}[Descent lemma for the averaged RF-EXTRA sequence]
\label{prop:rf-extra-descent}
Let $\bar H_k:=n^{-1}\sum_{i=1}^n H_i(X_{i,k})$, let $Q_k:=\bar X_k^\top \bar X_k-I_r$. Suppose Assumption~\ref{ass:smooth} holds, suppose moreover that \revise{$X_{i,k}\in \mathcal{B}$ for all $i=1,\dots,n$ and $\bar X_k\in\mathcal R$. Assume further that the remaining hypotheses in the second part of Lemma~\ref{lem:destiny-ingredients}, including its parameter assumptions, are satisfied at iteration $k$.} Assume \revise{$\beta\ge 56L_f^2$}, and let
$\revise{\alpha\le \min\left\{\frac{3}{8L_h},\frac{1}{4L_H}\right\},}$ then we have
\begin{equation}
\label{eq:descent-H-form}
h(\bar X_{k+1})
\le h(\bar X_k)-\revise{\frac{3\alpha}{16}}\|H(\bar X_k)\|_F^2+\Gamma_X(\alpha)\|\mathbf X_k-\overline{\mathbf X}_k\|_F^2,
\end{equation}
where
\[
\revise{\Gamma_X(\alpha):=\frac{17\alpha L_H^2}{8n}+\frac{\alpha^2L_hL_H^2}{n}}\le \Gamma_X^\sharp:=\frac{5L_H}{8n}.
\]
\end{lemma}

\begin{proof}
Write $H_k:=H(\bar X_k)$. By the second part of Lemma~\ref{lem:destiny-ingredients}, we obtain \revise{$\bar X_{k+1}\in\mathcal R$. Hence $\bar X_k,\bar X_{k+1}\in\mathcal R\subset\mathcal B$,} and \revise{$h$ is $L_h$-smooth on $\mathcal{B}$},
\begin{equation}
\label{origin}
h(\bar X_{k+1})
\le h(\bar X_k)-\alpha\langle \nabla h(\bar X_k),\bar H_k\rangle+\revise{\frac{\alpha^2L_h}{2}}\|\bar H_k\|_F^2.
\end{equation}
Split the inner product into
\[
-\alpha\langle \nabla h(\bar X_k)-H_k,\bar H_k\rangle-\alpha\langle H_k,\bar H_k\rangle.
\]
Since $\overline{X}_k\in \mathcal{R}$ implies $||\overline{X}_k||_2^2\leq\frac{7}{6},$ according to the definition of $L_f$, we have
\[
    \|\nabla h(\bar X_k)-H_k\|_F^2\leq \frac{9}{4}\|\bar X_k(\bar X_k^T\bar X_k-I_r)\|_F^2\le \frac{21}{8}L_f^2\|Q_k\|_F^2.
\]
Moreover,
\[
\|H_k-\bar H_k\|_F\le \frac{L_H}{\sqrt n}\|\mathbf X_k-\overline{\mathbf X}_k\|_F,
\qquad
\|\bar H_k\|_F^2\le 2\|H_k\|_F^2+\frac{2L_H^2}{n}\|\mathbf X_k-\overline{\mathbf X}_k\|_F^2.
\]
Next we bound $\langle \nabla h(\bar X_k)-H_k,\bar H_k\rangle, \langle H_k,\bar H_k\rangle$ and \revise{$\frac{\alpha^2L_h}{2}\|\bar H_k\|_F^2$} separately. For the first one, Young's inequality gives
\[
\begin{aligned}
-\alpha\langle \nabla h(\bar X_k)-H_k,\bar H_k\rangle
&\le \alpha\left(4\|\nabla h(\bar X_k)-H_k\|_F^2+\frac1{16}\|\bar H_k\|_F^2\right)\\
&\le \alpha\left[4\cdot \frac{21}{8}L_f^2\|Q_k\|_F^2+\frac{1}{16}\left(2\|H_k\|_F^2+\frac{2L_H^2}{n}\|\mathbf X_k-\overline{\mathbf X}_k\|_F^2\right)\right]\\
&=\frac{21}{2}\alpha L_f^2\|Q_k\|_F^2+\frac{\alpha}{8}\|H_k\|_F^2+\frac{\alpha L_H^2}{8n}\|\mathbf X_k-\overline{\mathbf X}_k\|_F^2.
\end{aligned}
\]
For the second term, we write
\[
\begin{aligned}
-\alpha\langle H_k,\bar H_k\rangle
&=-\alpha\langle H_k,\bar H_k-H_k\rangle-\alpha\|H_k\|_F^2\\
&\le \alpha\left(\frac18\|H_k\|_F^2+2\|\bar H_k-H_k\|_F^2\right)-\alpha\|H_k\|_F^2\\
&\le \alpha\left(\frac18\|H_k\|_F^2+\frac{2L_H^2}{n}\|\mathbf X_k-\overline{\mathbf X}_k\|_F^2\right)-\alpha\|H_k\|_F^2\\
&=-\frac{7\alpha}{8}\|H_k\|_F^2+\frac{2\alpha L_H^2}{n}\|\mathbf X_k-\overline{\mathbf X}_k\|_F^2.
\end{aligned}
\]
Finally, using the bound on $\|\bar H_k\|_F^2$, we obtain
\[
\begin{aligned}
\revise{\frac{\alpha^2L_h}{2}}\|\bar H_k\|_F^2
&\le \revise{\frac{\alpha^2L_h}{2}}\left(2\|H_k\|_F^2+\frac{2L_H^2}{n}\|\mathbf X_k-\overline{\mathbf X}_k\|_F^2\right)\\
&=\revise{\alpha^2L_h}\|H_k\|_F^2+\revise{\frac{\alpha^2L_hL_H^2}{n}}\|\mathbf X_k-\overline{\mathbf X}_k\|_F^2.
\end{aligned}
\]
Therefore, plugging the above three expressions into \eqref{origin}, we have
\[
h(\bar X_{k+1})
    \le h(\bar X_k)+\frac{21}{2}\alpha L_f^2\|Q_k\|_F^2+\Gamma_X(\alpha)\|\mathbf X_k-\overline{\mathbf X}_k\|_F^2+
\left(\revise{-\frac{3\alpha}{4}+\alpha^2L_h}\right)\|H_k\|_F^2.
\]
Because \revise{$\alpha\le 3/(8L_h)$}, the last coefficient is bounded above by $-3\alpha/8$. Since $\bar X_k\in\mathcal R$, Lemma~\ref{lem:destiny-ingredients} gives
\[
\|H_k\|_F^2\ge \beta\|Q_k\|_F^2,
\]
hence
\[
\frac{21}{2}\alpha L_f^2\|Q_k\|_F^2\le \frac{21L_f^2}{2\beta}\alpha\|H_k\|_F^2.
\]
Substituting this estimate yields
\[
h(\bar X_{k+1})
\le h(\bar X_k)-\left(\frac38-\frac{21L_f^2}{2\beta}\right)\alpha\|H_k\|_F^2+
\Gamma_X(\alpha)\|\mathbf X_k-\overline{\mathbf X}_k\|_F^2,
\]
\revise{Because $\beta\ge 56L_f^2$, the coefficient satisfies $\frac38-\frac{21L_f^2}{2\beta}\ge \frac{3}{16}$, which gives \eqref{eq:descent-H-form}.} Since \revise{$\alpha\le \min\left\{\frac{3}{8L_h},\frac{1}{4L_H}\right\}$}, then
\[
\Gamma_X(\alpha)\le \frac{17\alpha L_H^2}{8n}+\frac{3\alpha L_H^2}{8n}=\frac{5\alpha L_H^2}{2n}\leq \frac{5L_H}{8n}=\Gamma_X^\sharp.
\]
\end{proof}

\begin{theorem}[Exact $\mathcal O(1/K)$ convergence rate for the averaged iterate]
\label{thm:rf-extra-rate}
\revise{Assume Assumptions~\ref{ass:smooth} and \ref{ass:network}. Let $\delta_S>0$ satisfy}
\[
\revise{\delta_S\le c_1\min\left\{\frac{1}{1+\sigma_2(W)},\frac{\sqrt n}{L_H}\right\},}
\]
\revise{and suppose $\bar X_0\in\mathcal R$ and $(\mathbf X_0,\mathbf s_0)\in\mathcal N_S(\delta_S)$. Assume that $\beta$ satisfies}
\[
\revise{\beta\ge \max\left\{56L_f^2,\frac{6+21C_0}{5},\,12\sqrt2\left(M_g+1\right)\right\}.}
\]
\revise{Let}
\[
\begin{aligned}
\Delta_h&:=h(\bar X_0)-\inf_{X\in\mathcal R} h(X),
\qquad
\Gamma_X^\sharp:=\frac{5L_H}{8n},
\qquad
C_{P,0}^\sharp:=\|\mathbf X_0-\overline{\mathbf X}_0\|_F^2+C_{P,0},\\
B_H&:=c_2L_H\sqrt n,
\qquad
 a:=\frac{1+\sigma}{2}.
\end{aligned}
\]
If
\begin{equation}
\label{eq:alpha-absorb-range}
\revise{\alpha\le \bar\alpha
:=\min\left\{\frac{1}{4L_H},\frac{1}{4\beta L_b},\frac{1}{4\beta},\frac{c_1(1-\sigma)}{8c_2L_H},\sqrt{\frac{(1-a)\delta_S}{B_HM_H}},\frac{1}{M_H+1},\frac{3}{8L_h},\frac{3}{32\Gamma_X^\sharp C_{P,1}},1\right\},}
\end{equation}
then, for every integer $K\ge0$,
\begin{equation}
\label{eq:H-rate-ergodic}
\frac1{K+1}\sum_{k=0}^{K}\|H(\bar X_k)\|_F^2
\le \revise{\frac{32\bigl(\Delta_h+\Gamma_X^\sharp C_{P,0}^\sharp\bigr)}{3\alpha (K+1)}}.
\end{equation}
Consequently,
\begin{equation}
\label{eq:stationarity-components-rate}
\frac1{K+1}\sum_{k=0}^{K}\|G(\bar X_k)\|_F^2
\le \revise{\frac{32\bigl(\Delta_h+\Gamma_X^\sharp C_{P,0}^\sharp\bigr)}{3\alpha (K+1)}}.
\end{equation}
\begin{equation}
\label{eq:feasibility-rate}
\frac1{K+1}\sum_{k=0}^{K}\|\bar X_k^\top\bar X_k-I_r\|_F^2
\le \revise{\frac{32\bigl(\Delta_h+\Gamma_X^\sharp C_{P,0}^\sharp\bigr)}{3\alpha\beta (K+1)}}.
\end{equation}
Hence the RF-EXTRA averaged iterate attains an exact $\mathcal O(1/K)$ convergence rate.
\end{theorem}

\begin{proof}
\revise{Repeated application of Lemma~\ref{lem:star_inv} gives $\bar X_k\in\mathcal R$ and $(\mathbf X_k,\mathbf s_k)\in\mathcal N_S(\delta_S)$ for all $k\ge0$. Hence $\bar X_{k+1}\in\mathcal R$ as well. By norm equivalence,}
\[
\revise{\|\mathbf X_k-\overline{\mathbf X}_k\|_F\le \|\mathbf p_k\|_F\le \frac{1}{c_1}\|\mathbf p_k\|_S\le \frac{\delta_S}{c_1}\le \frac{1}{1+\sigma_2(W)}\le 1.}
\]
\revise{Since $\bar X_k\in\mathcal R$, we have $\|\bar X_k\|_F\le \sqrt{7r/6}$; therefore, for every $i$,}
\[
\revise{\|X_{i,k}\|_F\le \|\bar X_k\|_F+\|X_{i,k}-\bar X_k\|_F\le \sqrt{7r/6}+\|\mathbf X_k-\overline{\mathbf X}_k\|_F\le \sqrt{7r/6}+1,}
\]
\revise{so $X_{i,k}\in\mathcal B$ for all $i,k$. Moreover, $\mathcal R\subset\mathcal B$, hence $\bar X_k,\bar X_{k+1}\in\mathcal B$. Therefore Lemma~\ref{prop:rf-extra-descent} applies at every step.} Summing \eqref{eq:descent-H-form} from $k=0$ to $K$ and using $\Gamma_X(\alpha)\le \Gamma_X^\sharp$ yield
\[
h(\bar X_{K+1})
\le h(\bar X_0)-\revise{\frac{3\alpha}{16}}\sum_{k=0}^{K}\|H(\bar X_k)\|_F^2+\Gamma_X^\sharp\sum_{k=0}^{K}\|\mathbf X_k-\overline{\mathbf X}_k\|_F^2.
\]
Now split the consensus sum as
\[
\sum_{k=0}^{K}\|\mathbf X_k-\overline{\mathbf X}_k\|_F^2
=\|\mathbf X_0-\overline{\mathbf X}_0\|_F^2+\sum_{k=0}^{K-1}\|\mathbf X_{k+1}-\overline{\mathbf X}_{k+1}\|_F^2.
\]
Applying Lemma~\ref{lem:joint-error-sum} to the second term and discarding the nonnegative $\|\mathbf s_{k+1}-\overline{\mathbf s}_{k+1}\|_F^2$ part give
\[
\sum_{k=0}^{K}\|\mathbf X_k-\overline{\mathbf X}_k\|_F^2
\le C_{P,0}^\sharp+\alpha^2 C_{P,1}\sum_{k=0}^{K-1}\|H(\bar X_k)\|_F^2
\le C_{P,0}^\sharp+\alpha^2 C_{P,1}\sum_{k=0}^{K}\|H(\bar X_k)\|_F^2.
\]
Substituting this estimate gives
\[
h(\bar X_{K+1})
\le h(\bar X_0)-\revise{\left(\frac{3\alpha}{16}-\Gamma_X^\sharp\alpha^2 C_{P,1}\right)}
\sum_{k=0}^{K}\|H(\bar X_k)\|_F^2+\Gamma_X^\sharp C_{P,0}^\sharp.
\]
By \eqref{eq:alpha-absorb-range},
\[
\revise{\Gamma_X^\sharp\alpha C_{P,1}\le \frac{3}{32}},
\qquad\text{hence}\qquad
\revise{\frac{3\alpha}{16}-\Gamma_X^\sharp\alpha^2 C_{P,1}\ge \frac{3\alpha}{32}.}
\]
Since $h(\bar X_{K+1})\ge \inf_{X\in\mathcal R}h(X)$, we conclude that
\[
\revise{\frac{3\alpha}{32}}\sum_{k=0}^{K}\|H(\bar X_k)\|_F^2
\le h(\bar X_0)-h(\bar X_{K+1})+\Gamma_X^\sharp C_{P,0}^\sharp
\le \Delta_h+\Gamma_X^\sharp C_{P,0}^\sharp.
\]
Dividing by $K+1$ proves \eqref{eq:H-rate-ergodic}. Finally, Lemma~\ref{lem:destiny-ingredients} gives
\[
\|H(\bar X_k)\|_F^2\ge \|G(\bar X_k)\|_F^2+\beta\|\bar X_k^\top\bar X_k-I_r\|_F^2,
\]
hence \eqref{eq:stationarity-components-rate} and \eqref{eq:feasibility-rate} follow immediately from \eqref{eq:H-rate-ergodic}.
\end{proof}

The convergence proof preserves the averaged iterate and the joint error introduced earlier, while replacing a Frobenius-norm contraction argument by an equivalent-norm analysis of the EXTRA subsystem. As a result, the final estimate is obtained through a transparent descent-and-accumulation argument: averaged descent provides the main decrease, the joint-error recursion quantifies the network-induced perturbation, and the summed joint-error bound makes it possible to conclude under explicit conditions on $\alpha$ and $\beta$.

\begin{remark}
\revise{The stepsize $\alpha$ is required to be sufficiently small, but it can still be chosen as a positive constant independent of $K$. The conditions on the penalty parameter $\beta$ are explicit and are used only to keep the averaged penalty step inside the neighborhood where the Stiefel geometry is controlled. The quantity $\delta_S$ may need to be small; however, this can be achieved by initializing all nodes with the same iterate, in which case the initial disagreement vanishes.}
\end{remark}

\section{Experiments}
\label{sec:experiments}

In this section, we compare RF-EXTRA with four decentralized baselines: DPRGD \cite{Deng2025DecenProjRG}, DPRGT \cite{Wang2025DecenPGTRiem}, DESTINY \cite{Wang2022StiefelAL}, and REXTRA \cite{Wu2025RiemannianEC}. Our implementation follows the experimental setting of \cite{Wu2025RiemannianEC}.

\subsection{Decentralized principal component analysis}
\label{subsec:dpca}

The decentralized principal component analysis (PCA) problem seeks a common low-dimensional subspace that preserves the maximal variation of data distributed across multiple agents. It can be formulated as
\begin{equation}
\label{eq:decentralized-pca}
\min_{\mathbf X\in \mathcal M^n} -\frac{1}{2n}\sum_{i=1}^n \operatorname{tr}(X_i^\top A_i^\top A_i X_i),
\qquad \text{s.t.}\qquad X_1=\cdots=X_n,
\end{equation}
where $\mathcal M^n:=\mathrm{St}(d,r)^n$, $\mathrm{St}(d,r)$ is the Stiefel manifold, $n$ is the number of agents, and $A_i\in\mathbb R^{m_i\times d}$ is the local data matrix stored at agent $i$. As in \cite{Wu2025RiemannianEC}, if $X^*$ solves \eqref{eq:decentralized-pca}, then $X^*Q$ is also a solution for any orthogonal matrix $Q\in\mathbb R^{r\times r}$. We therefore measure the distance to a reference solution by
\[
 d_s(X,X^*):=\min_{Q^\top Q=QQ^\top=I_r}\|XQ-X^*\|.
\]

\subsubsection{Synthetic dataset}

We fix $m_1=\cdots=m_n=1000$, $d=10$, $r=5$, and $n=8$. A Gaussian matrix $B\in\mathbb R^{1000n\times d}$ is generated and decomposed as $B=U\Sigma V^\top$. We then set $\widetilde\Sigma=\operatorname{diag}(\xi^j)$ with $\xi\in(0,1)$ and construct $A=U\widetilde\Sigma V^\top$, whose rows are partitioned uniformly across the $n$ agents. The first $r$ columns of $V$ form a solution to \eqref{eq:decentralized-pca}. In all experiments we use $\xi=0.8$, and set a Riemannian gradient stopping threshold $\|\operatorname{grad} f(\bar{X}_k)\|<10^{-8}$. For the PCA problem, the effective scaling is chosen as
\[
\alpha=\frac{\hat\beta n}{\sum_{i=1}^n m_i}.
\]
All algorithms use constant step sizes. For the synthetic comparison below, the step size is selected from
\[
\hat{\beta}\in\{1,2,4,6,8\}\times\{10^{-5},10^{-4},10^{-3},10^{-2}\}.
\]

We first study the robustness of RF-EXTRA with respect to graph topology and the internal parameter $\beta$. Figure~\ref{fig:rfextra-robustness} shows that the same qualitative behavior is preserved across ring, star, and ER graphs of different densities, and also across a wide range of internal $\beta$ values. This supports the claim that the retraction-free recursion is reasonably stable under moderate changes in both network structure and internal tuning.

\begin{figure}[t]
\centering
\includegraphics[width=0.9\textwidth]{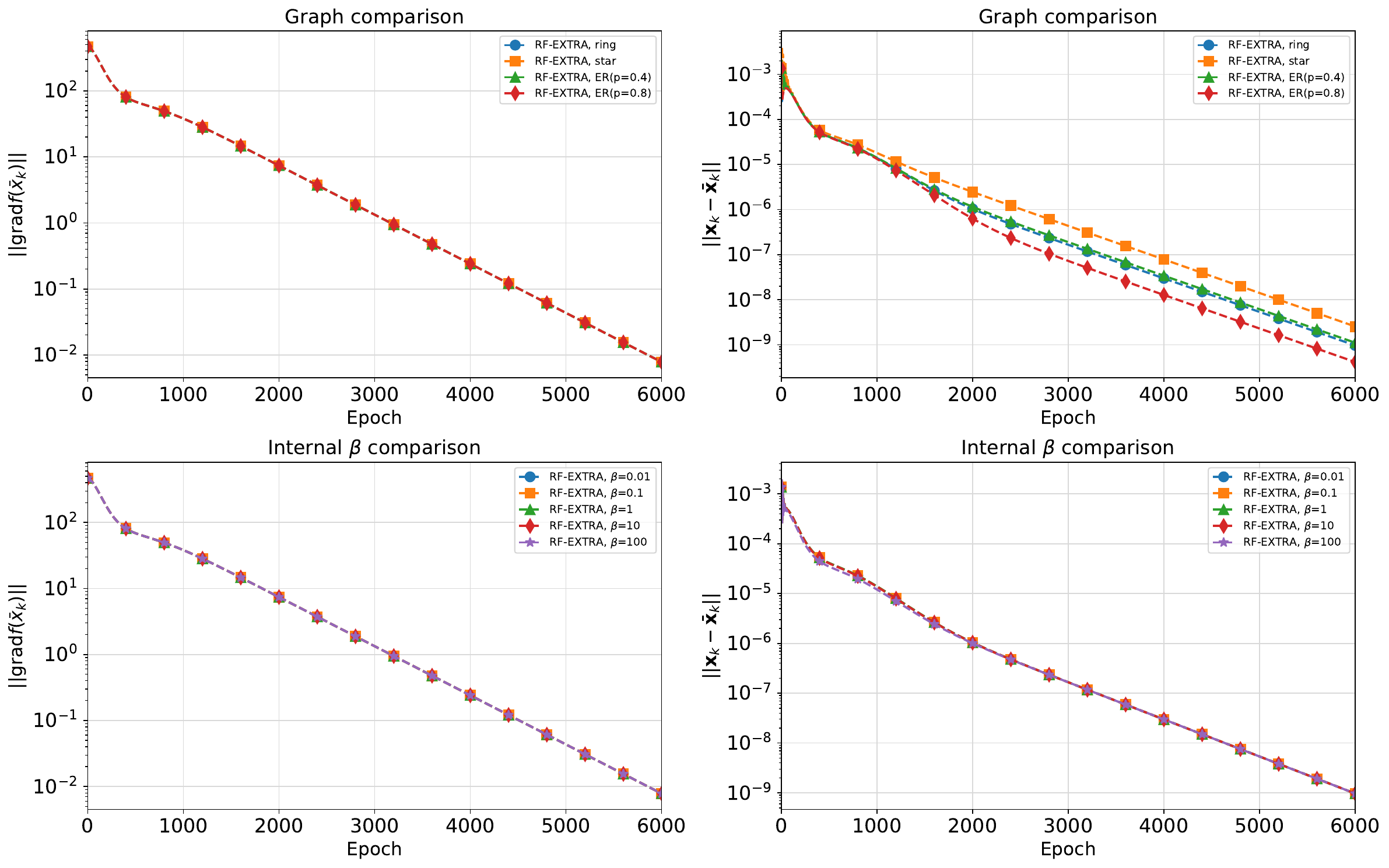}
\caption{Synthetic decentralized PCA: robustness of RF-EXTRA with respect to graph topology and internal $\beta$. Top row: stationarity and consensus trajectories under ring, star, ER$(0.4)$, and ER$(0.8)$ graphs. Bottom row: the same quantities when the internal parameter $\beta$ varies in $\{0.01,0.1,1,10,100\}$ on the ring graph.}
\label{fig:rfextra-robustness}
\end{figure}

We next compare the wall-clock efficiency of RF-EXTRA and REXTRA under matched graph settings. Table~\ref{tab:synthetic-wallclock} shows that, once both methods are stably tuned, RF-EXTRA is typically slightly faster than REXTRA, which is consistent with the additional cost of retraction operations in the latter.

\begin{table}[t]
\centering
\caption{Final wall-clock time (seconds) of RF-EXTRA and REXTRA on ER graphs. Smaller is better.}
\label{tab:synthetic-wallclock}
\begin{tabular}{lccc}
\toprule
Method / stepsize & $p=0.4$ & $p=0.6$ & $p=0.8$\\
\midrule
RF-EXTRA, $\hat\beta=0.3$   & 0.617 & 0.583 & 0.588\\
RF-EXTRA, $\hat\beta=0.05$  & 3.388 & 3.268 & 3.384\\
RF-EXTRA, $\hat\beta=0.008$ & 20.433 & 20.716 & 22.598\\
REXTRA, $\hat\beta=0.3$     & 0.609 & 0.645 & 0.670\\
REXTRA, $\hat\beta=0.05$    & 3.613 & 3.574 & 3.570\\
REXTRA, $\hat\beta=0.008$   & 23.447 & 22.364 & 23.669\\
\bottomrule
\end{tabular}
\end{table}

After these focused robustness and efficiency checks, we turn to the standard synthetic decentralized PCA task under the same ER$(0.6)$ graph, seed, and epoch budget for all methods. Figure~\ref{fig:synthetic-epoch} reports the best run of RF-EXTRA, DESTINY, DPRGT, DPRGD, and REXTRA under this unified search space. RF-EXTRA is competitive with the strongest baselines across the communication budget, and its performance is close to that of REXTRA. Combined with the wall-clock comparison in Table~\ref{tab:synthetic-wallclock}, this indicates that RF-EXTRA achieves similar empirical behavior while avoiding the extra retraction cost.

\begin{figure}[t]
\centering
\includegraphics[width=\textwidth]{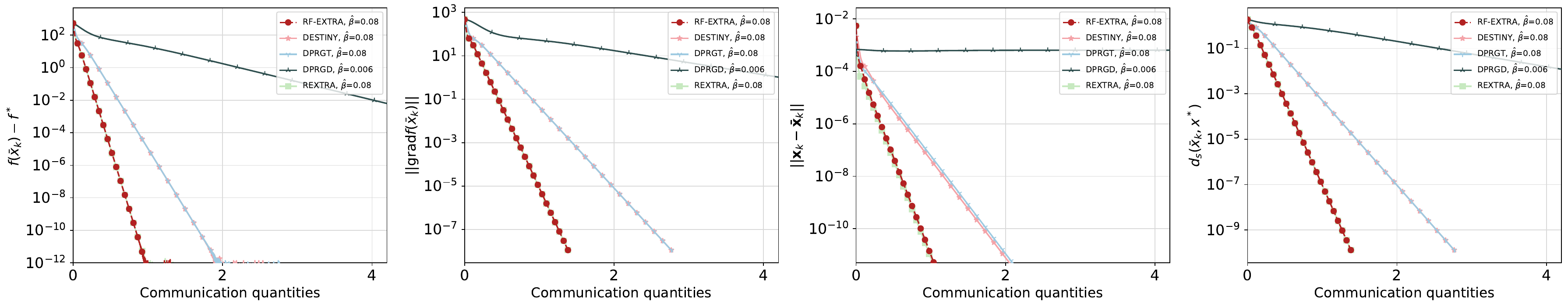}
\caption{Synthetic decentralized PCA on ER$(0.6)$ versus communication quantities. Each method uses its best step size selected from $\{1,2,4,6,8\}\times\{10^{-5},10^{-4},10^{-3},10^{-2}\}$. Under this matched search space, RF-EXTRA, DESTINY, DPRGT, and REXTRA all select $\hat\beta=0.08$, while DPRGD selects $\hat\beta=0.006$.}
\label{fig:synthetic-epoch}
\end{figure}

\subsubsection{MNIST dataset}

To further evaluate the practical behavior of RF-EXTRA on real data, we conduct experiments on the MNIST dataset. The dataset contains $60{,}000$ handwritten digit images of size $28\times 28$, which are used to generate the local matrices $A_i$. We first normalize the pixel values by dividing by $255$ and then randomly partition the data into $n=8$ agents with equal numbers of samples. As a result, each agent holds a local matrix $A_i$ of dimension $\frac{60000}{n}\times 784$. We compute the top $r=2$ principal components with ambient dimension $d=784$, use a Riemannian gradient stopping threshold $\|\operatorname{grad} f(\bar{X}_k)\|<10^{-6}$, and choose the effective scaling
\[
\alpha=\frac{\hat\beta}{60000}.
\]
Similar to the synthetic PCA setting, all methods use constant step sizes, and we select the best available step size from the grid
\[
\hat{\beta}\in\{1,2,6\}\times\{10^{-4},10^{-3},10^{-2}\}.
\]
Among the runs already completed in this grid, the strongest available setting is $\hat\beta=0.06$ for RF-EXTRA, DESTINY, and DPRGT, while the best available DPRGD run uses $\hat\beta=0.02$. Accordingly, Figure~\ref{fig:mnist-comm} reports RF-EXTRA, DESTINY, and DPRGT at $\hat\beta=0.06$ together with DPRGD at $\hat\beta=0.02$.

Figure~\ref{fig:mnist-comm} shows that RF-EXTRA and DESTINY perform similarly and both improve on DPRGT in the later stage, while DPRGD remains clearly weaker within the same communication budget.

\begin{figure}[t]
\centering
\includegraphics[width=\textwidth]{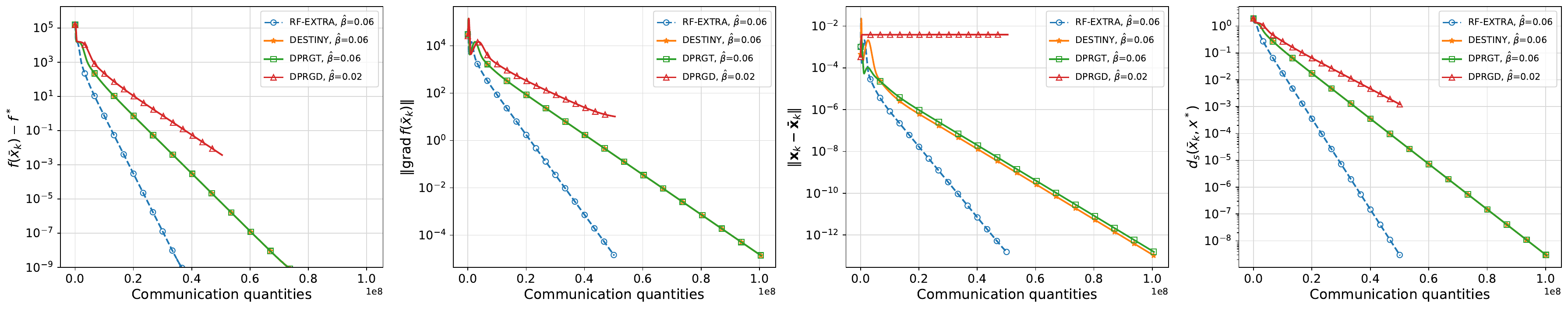}
\caption{Decentralized PCA on the MNIST dataset versus communication quantities. RF-EXTRA, DESTINY, and DPRGT are plotted at $\hat\beta=0.06$, while DPRGD is plotted at its best existing MNIST noMPI step size $\hat\beta=0.02$ among the currently available runs from $\{1,2,6\}\times\{10^{-4},10^{-3},10^{-2}\}$.}
\label{fig:mnist-comm}
\end{figure}

\subsection{Decentralized low-rank matrix completion}
\label{subsec:lrmc}

Low-rank matrix completion (LRMC) aims to recover a low-rank matrix $A\in\mathbb R^{d\times T}$ from a subset of its observed entries. Let $\Omega$ denote the index set of observed entries. The rank-$r$ LRMC problem can be written as
\[
\min_{X\in \mathrm{Gr}(d,r),\;V\in \mathbb R^{r\times T}} \frac12\bigl\|\mathcal P_{\Omega}\odot(XV-A)\bigr\|^2,
\]
where $\mathrm{Gr}(d,r)$ is the Grassmann manifold, $\odot$ denotes the Hadamard product, and $\mathcal P_{\Omega}$ is the observation mask. In the decentralized setting, the partially observed matrix is partitioned column-wise across $n$ agents into local blocks $A_1,\ldots,A_n$. Using the standard representation of the Grassmann manifold through the Stiefel manifold, the problem can be reformulated as a decentralized optimization problem over $\mathrm{St}(d,r)$,
\[
\min_{X_i\in \mathrm{St}(d,r)} \frac12\sum_{i=1}^n \bigl\|\mathcal P_{\Omega_i}\odot\bigl(X_iV_i(X_i)-A_i\bigr)\bigr\|^2,
\qquad \text{s.t.}\qquad X_1=\cdots=X_n,
\]
where $\Omega_i$ is the local observation pattern and $V_i(X_i)$ is the corresponding least-squares factor. Following \cite{Wu2025RiemannianEC}, we consider a synthetic setup with $T=1000$, $d=100$, $r=5$, and $n=8$, where the ground-truth low-rank matrix is generated from Gaussian factors and the observation mask is sampled with rate $\mu=r(d+T-r)/(dT)$. More precisely, if $L\in\mathbb R^{d\times r}$ and $R\in\mathbb R^{r\times T}$ are Gaussian factors, then the target matrix is generated as
\[
A = LR + 10^{-3}E,
\]
where $E\in\mathbb R^{d\times T}$ has i.i.d. standard normal entries. Thus, the LRMC instance is mildly perturbed from an exactly rank-$r$ model. We adopt a ring graph to model the communication network among agents. For LRMC, we use the effective scaling
\[
\alpha=\hat\beta\times n.
\]
All algorithms are implemented with constant step sizes, and the candidate step sizes are selected from the grid
\[
\hat\beta\in \{1.25,2.5,6.25,10\}\times\{10^{-5},10^{-4},10^{-3}\}.
\]
We run each method for at most $1500$ epochs and terminate early if $\|\operatorname{grad} f(\bar{X}_k)\|<10^{-6}$.

Because the perturbed LRMC instance does not admit a practically meaningful exact reference solution, we only report the stationarity and Euclidean consensus panels in the revised communication plot. Figure~\ref{fig:lrmc-comm} shows that RF-EXTRA reaches low stationarity and consensus levels more rapidly than DESTINY and DPRGT in the plotted regime, while DPRGD remains less competitive within the same budget. Figure~\ref{fig:lrmc-epoch} further compares representative RF-EXTRA and DESTINY step sizes on the same communication axis and shows the same qualitative trend.

\begin{figure}[t]
\centering
\includegraphics[width=0.82\textwidth]{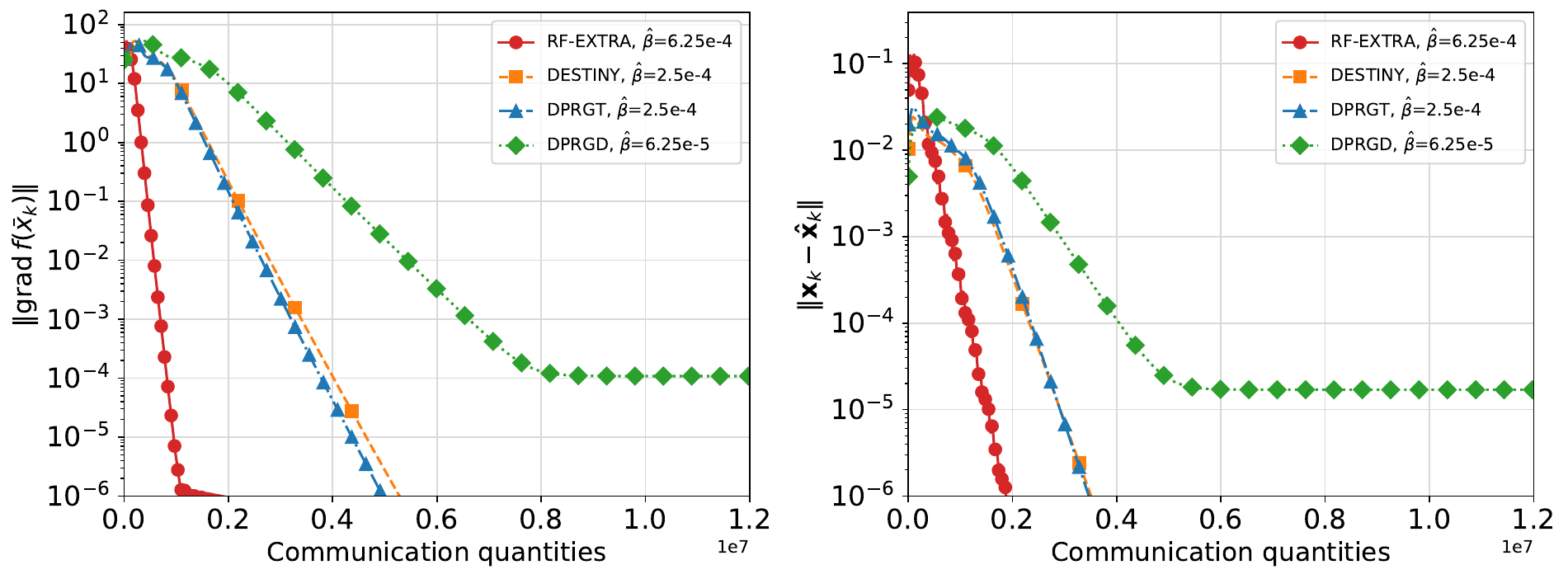}
\caption{Decentralized LRMC on the ring graph versus communication quantities. Only the stationarity and Euclidean consensus panels are shown, since the perturbed instance does not provide an exact-solution reference panel of comparable value. For each method, the step size is selected by the earliest epoch at which $\|\operatorname{grad} f(\bar{x}_k)\|<10^{-6}$.}
\label{fig:lrmc-comm}
\end{figure}

\begin{figure}[t]
\centering
\includegraphics[width=0.82\textwidth]{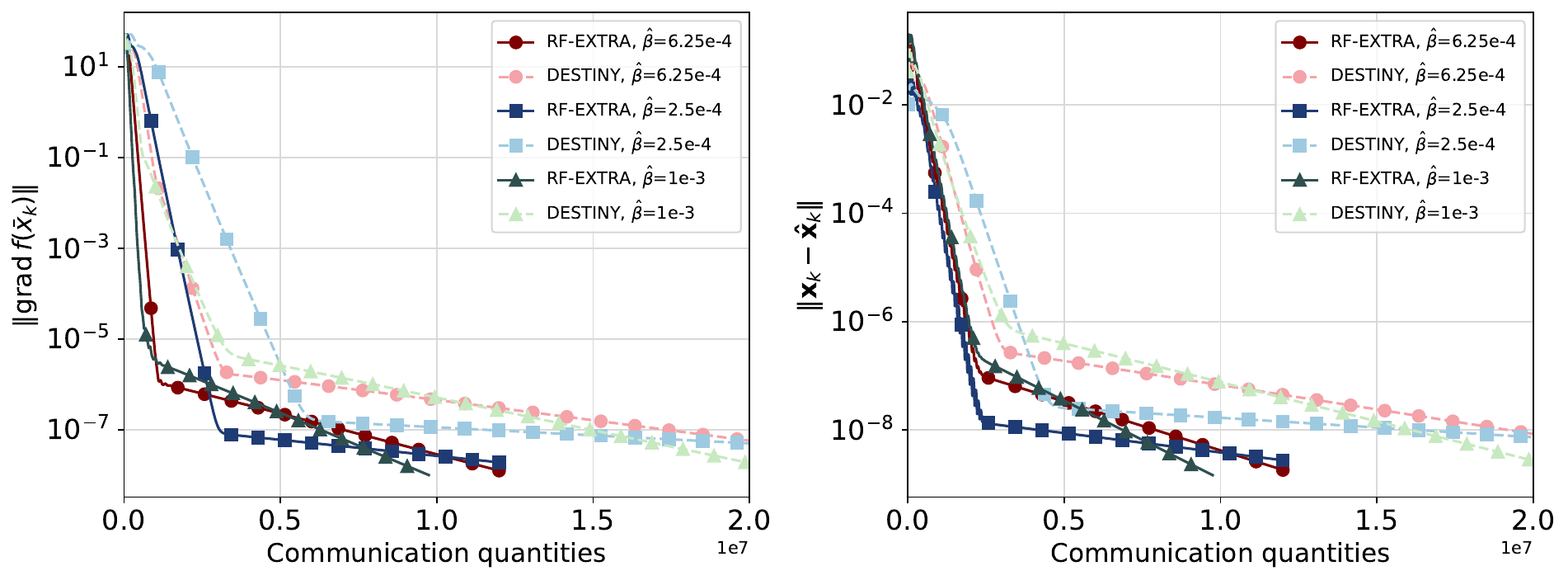}
\caption{Decentralized LRMC on the ring graph versus communication quantities for representative RF-EXTRA and DESTINY step sizes. We again retain only the stationarity and Euclidean consensus panels.}
\label{fig:lrmc-epoch}
\end{figure}

\section{Conclusion}
\label{sec:conclusion}

We proposed RF-EXTRA, a decentralized retraction-free method for optimization on the Stiefel manifold. The method combines an ambient-space surrogate for the orthogonality-constrained local models with an EXTRA-based primal-dual recursion, thereby preserving a simple decentralized communication structure while avoiding explicit per-iteration retractions. On the theoretical side, the analysis is built on the averaged iterate and the joint error $(\mathbf X_k-\overline{\mathbf X}_k,\mathbf s_k-\overline{\mathbf s}_k)$. By establishing a contractive recursion for the joint error under an equivalent norm and then comparing the different terms in the descent analysis, we obtain an exact $\mathcal O(1/K)$ convergence guarantee for the averaged iterate. The experiments indicate that RF-EXTRA is a competitive retraction-free decentralized solver for optimization problems on the Stiefel manifold. On PCA and low-rank matrix completion, it delivers favorable empirical performance and communication efficiency relative to the reported baselines. These results support retraction-free decentralized correction as a promising direction for large-scale optimization on matrix manifolds.

\section*{Acknowledgements}
We gratefully acknowledge ReasFlow \cite{reasflowteam2026reasflow}, a reasoning-centric scientific discovery assistant, for its substantial contributions to the preparation of this paper. A significant portion of the work, including the literature review, mathematical proofs, numerical experiments, and the initial manuscript draft, was generated automatically with the assistance of ReasFlow. The authors’ contributions lay primarily in identifying the research problem, proposing the high-level algorithmic design, articulating the key ideas underlying the mathematical proofs, specifying the methodology and requirements for the numerical experiments, and polishing the manuscript to meet the standards required for submission. In particular, the authors devoted considerable effort to verifying the correctness of the mathematical proofs and refining the resulting arguments.

\bibliographystyle{unsrtnat}
\bibliography{references}

\end{document}